\documentclass[review,hidelinks,onefignum,onetabnum]{siamart220329}

\usepackage{amsfonts}
\usepackage{graphicx}
\usepackage{epstopdf}
\usepackage{algorithmic}
\usepackage{ctable}
\ifpdf
  \DeclareGraphicsExtensions{.eps,.pdf,.png,.jpg}
\else
  \DeclareGraphicsExtensions{.eps}
\fi

\newsiamremark{remark}{Remark}
\newsiamremark{hypothesis}{Hypothesis}
\crefname{hypothesis}{Hypothesis}{Hypotheses}
\newsiamthm{claim}{Claim}

\headers{Hybrid LSMR algorithms}{Yangfei Yang}

\title{Hybrid LSMR algorithms for large-scale general-form regularization
}

\author{Yanfei Yang\thanks{Department of Mathematics and Computer Science,
Zhejiang A\&F University, Hangzhou 311300, China
(\email{yang@zafu.edu.cn}).}}

\usepackage{amsopn}

\begin{document}
\nolinenumbers
\maketitle

\begin{abstract}
The hybrid LSMR algorithm is proposed
for large-scale general-form regularization.
It is based on a Krylov subspace projection method
where the matrix $A$ is first projected onto a subspace,
typically a Krylov subspace,
which is implemented via the Golub-Kahan 
bidiagonalization process applied to $A$,
with starting vector $b$.
Then a regularization term is employed to the projections.
Finally, an iterative algorithm is exploited to solve a least squares problem with constraints.
The resulting algorithms are called
the {hybrid LSMR algorithm}.
At every step, we exploit LSQR algorithm to solve the 
inner least squares problem,
which is proven to become better conditioned as the number of $k$ increases,
so that the LSQR algorithm converges faster.
Numerical experiments demonstrate that the regularized solution by our new hybrid LSMR algorithm is as accurate as that obtained by the JBDQR algorithm, which is based on joint bidiagonalization, and that our hybrid algorithm is significantly more efficient than JBDQR.
\end{abstract}



\section{Introduction}
Consider the large-scale linear ill-posed problem of the form
\begin{equation}\label{eq1}
\min_{x\in \mathbb{R}^{n}}\|Ax - b\| \quad {\rm or}
\quad Ax = b, \quad A \in \mathbb{R}^{m \times n}, \quad b \in \mathbb{R}^{m},
\end{equation}
where the norm $\|\cdot\|$ is the 2-norm of a vector or matrix,
and $A$ is ill-conditioned with its singular values
decaying and centered at zero without
a noticeable gap,  and the right-hand side
$b = b_{true}+e$ is assumed to be contaminated
by a Gaussian white noise $e$, where $b_{true}$
is the noise-free right-hand side and $\|e\| < \|b_{true}\|$.
Without loss of generality, assume that
$Ax_{true} = b_{true}$.

Discrete ill-posed problems of the form \eqref{eq1} may be
derived from the discretization of linear ill-posed problems,
such as the Fredholm integral equation of the first kind,
\begin{equation*}
Kx=(Kx)(s)=\int_{\Omega}k(s,t)x(t)dt=g(s), s\in \Omega\subset \mathbb{R}^q,
\end{equation*}
where the kernel $k(s,t)\in L^2(\Omega\times \Omega)$ and
$g(s)$ are known functions, while $x(t)$ is
the unknown function to be sought.
These problems arise in various scientific and applications areas,
including mining engineering and
astronomy, biomedical sciences, and geoscience; see, e.g.
\cite{aster, berisha, engl96,epstein2007,haber2014, hansen10, kirsch11,  miller,nat01,vogel02}.
Due to the presence of the noise $e$ and the high ill-conditioning
of $A$, he naive solution
$x_{naive}=A^{\dag}b$ to \eqref{eq1} typically
is meaningless, as it
bears no relation to the true solution
$x_{true}=A^{\dag}b_{true}$,
where $\dag$ denotes the Moore-Penrose inverse of a matrix.
To compute a meaningful solution,
it is necessary to exploit \emph{regularization}
to overcome the inherent instability of ill-posed
problems. The key idea of regularization 
is to replace the original problem
\eqref{eq1} with a modified, more stable problem, 
allowing for the computation of a regularized solution that approximates the true solution.
There are various regularization techniques taking many forms;
see e.g.,\cite{chung,hansen98,hansen10} for more details.

A simple and popular regularization
is \emph{iterative regularization} which solves the underlying problem \eqref{eq1} iteratively.
In this setting, an iterative method is applied directly to
\begin{equation*}
\min_{x\in \mathbb{R}^{n}}\|Ax - b\|
\end{equation*}
and regularization is obtained by
terminating early.
It is well known that the
iterative method exhibits semiconvergence
on ill-posed problems, with
errors
decreasing initially but at some point beginning to
increase since the small singular values of $A$ start to amplify noise.
Therefore, for iterative regularization,
a vital and nontrivial
task is to select a good stopping iteration.

Another popular and well-established
way is \emph{hybrid regularization},
which computes an approximation of $x_{true}$
by solving 
\begin{equation}\label{gen1}
\min_{x\in\mathbb{S}}\|Lx\| \ {\rm subject \ to } \
\mathbb{S}=\{x|\|Ax-b\|\leq\tau \|e\|\}
\end{equation}
or 
\begin{equation}\label{gen2}
\min_{x\in\mathbb{R}^n}\{\|Ax-b\|^2+\lambda^2\|Lx\|^2\},
\end{equation}
where $\tau\approx 1$, and $L\in \mathbb{R}^{p\times n}$
is a regularization matrix,
usually a discrete approximation of some derivative operators,
and $\lambda>0$ is the regularization parameter
that controls the amount of regularization to balance
the fitting term and the regularization term; see e.g.,
\cite{hansen98,hansen10,tikhonov63}.
Problem \eqref{gen1} is discrepancy principle based general-form regularization
which is equivalent to the general-form Tikhonov regularization \eqref{gen2}.
The solution to (\ref{gen2}) is unique for a given $\lambda>0$ when
$N(A)\cap N(L)={0}$
where $N(\cdot)$ denotes the null space of a matrix.
When $L = I_n$, with $I_n$ being the $n \times n$ identity matrix,
problem (\ref{gen1}) and problem \eqref{gen2}
are said to be in standard form.


It is worth noting that regularization in norms other than the 2-norm is also important, as some ill-posed problems are based on underlying mathematical models that are not linear.
Consider general optimizations of the form
\begin{equation*}
\min_{x\in \mathbb{S}}\mathcal{R}(Lx)  \ {\rm subject \ to } \
\mathbb{S }=\{x|\mathcal{J}(Ax-b) \leq \tau \|e\| \}
\end{equation*}
or
\begin{equation*}
\min_{x\in\mathbb{R}^n}\{\mathcal{J}(Ax-b)+\lambda^2\mathcal{R}(Lx)\},
\end{equation*}
where $\mathcal{J}$ is a fit-to-data term and $\mathcal{R}$ is a regularization term.
see e.g., \cite{gazzola2014} and \cite[pp.~120-121]{hansen98}.
Solving these problems need to use nonlinear optimization methods,
however, many of these methods require solving a subproblem
with an approximate linear model \cite{chung,gazzola2014,hansen98}.

In this paper, we consider the hybrid LSMR method.
The basic idea of the hybrid method is first to exploit LSMR
to solve the underlying problem and meanwhile project it 
onto a relatively stable problem
and then employ a regularization term to the projections.
Instead of \eqref{gen2}, we focus our discussion on \eqref{gen1}, but
the hybrid method we present is fairly general and can be used to solve
a variety of combinations of fit-to-data term and regularization term.

Most of the existing hybrid algorithms for solving large-scale ill-posed problems in general-form regularization solve \eqref{gen2}; see, e.g., \cite{novati2014adaptive,Novati2014gcv,bazan2014,gazzola2013multi,gazzola2014,yang} for more details.
Though, mathematically, \eqref{gen1} is equivalent to \eqref{gen2},
numerically, it is totally different between hybrid methods for solving
\eqref{gen1} and \eqref{gen2}.
The hybrid method for solving \eqref{gen1} has nothing to do with regularization parameter because the number of iteration plays the role of the regularization parameter,
and it is the first-projection-then-regularization method.
Moreover, the first-projection-then-regularization method for \eqref{gen1}
is definitely different from the first-regularization-then-projection one.
For solving \eqref{gen1}, one of the first-regularization-then-projection algorithms
is JBDQR presented by Jia and Yang; see \cite{jiayang2020} for more details.
The hybrid method for solving \eqref{gen2}, however, 
needs to determine the optimal regularization parameter
at every step. Unfortunately, not like the underlying problem, the projected problems may not satify the Picard condition,
so it is difficult to determined the optimal regularization
parameter of the projection . 
Further more, it is difficult to prove the optimal regularization parameter of the projection to be
the best regularization parameter of the original problem unless
the Krylov subspace captures the singular values in their natural order,
starting with the largest, for details see e.g.,\cite{hansen98,hansen10}.
Only a few existing hybrid algorithms solve \eqref{gen1}.

Hansen \textit{et al.} \cite{hansen92} propose a 
modified truncated SVD (MTSVD) method for solving \eqref{gen1},
which is an alternative to the truncated GSVD method which solves \eqref{gen2}. 
The algorithm first computes the SVD of $A$ and 
then extracts the best rank-$k$ approximation to $A$ by truncating the SVD of $A$.
It solves a sequence of least squares problems by the adaptive QR factorization
until a best regularized solution is found. 
The number of the truncation $k$ plays the role of the regularization parameter.
Although this algorithm avoids computing the GSVD of the matrix 
pair $\{A, L\}$, it is not suitable for large-scale problems 
since computing the SVD of $A$ is infeasible for $A$ large.

Based on the joint bidiagonalization process (JBD) process,
Jia and Yang \cite{jiayang2020} propose the JBDQR algorithm for sloving \eqref{gen1}.
The JBD process is an inner-outer iterative process.
At each outer iteration, the JBD process needs to compute the solution of a large-scale linear least squares problem with the coeﬃcient matrix $(A^T, L^T )^T$ that is larger than the problem itself and supposed to be solved iteratively, called inner iteration. Fortunately, $(A^T, L^T )^T$ is generally well conditioned, as $L$ is typically so in applications \cite{hansen98,hansen10}. 
In these cases, the LSQR algorithm \cite{lsqr1982} can solve the least squares problems mentioned efficiently. Although the underlying solution subspaces generated by the process are legitimate, 
the overhead of computation of the methods based on the process may be extremely expensive.
Unfortunately, methods based on the JBD process can not avoid solving the large-scale inner least squares problem at every iteration. Finally, methods based on the process need to solve a large-scale least squares problem with the coeﬃcient matrix $(A^T , L^T )^T$ to form a regularized solution. Therefore, the method based on the JBD process may be substantially expensive.

Yang \cite{yang} presents the hybrid CGME and TCGME algorithms for solving \eqref{gen1}.
These algorithms first expolit the Krylov solver, CGME and TCGME, 
to solve the underlying problem \eqref{eq1},
and then apply the general-form regularization term to the 
projected problems which is generated by the Krylov solvers.
Finally, computing the solution to the projected problems with general-form regularization. 
The method, on the one hand, can take the full advantage of 
the Krylov subspace method to reduce the large problem to the small or medium problem.
On the other hand, it can retain the information as much as possible since 
it leaves the regularization term remaining unchanged.
However, CGME and TCGME only have partial regularization,
which means they may not capture all the needed dominant SVD components of $A$.
Therefore, it can not be expected to capture all the needed dominant GSVD of $\{A, L\}$,
which is the key to solve \eqref{gen1} and \eqref{gen2} successfully.
As a consequence, the accuracy of the hybrid CGME algorithm is not so high.



In this paper, we are concerned with the situation when $L$
is fairly general and large.
We describe a new projection-based algorithm to solve \eqref{gen1}.
First, we expolit the LSMR algorithm to solve \eqref{eq1}
and the matrix $A$ is
projected onto a sequence of nested Krylov subspaces.
Then a general-form regularization term is applied to the projections,
Finally we solve projected problems with a general-term regularization.
At every step, we use LSQR to solve the inner linear least squares problems that result, which are proven to become better conditioned as $k$ increases,
so the iterative algorithm converges faster.
In theory, the inner linear least squares problems need to be solved accurately.
We prove how to choose the stopping
tolerance for the iterative algorithm
to solve the inner least squares problems,
in order to guarantee that the regularized solutions when
the inner least squares problems are solved by iterative
methods have the same accuracy as the ones when the inner least squares problems are solved exactly. The resulting algorithm
is called the hybrid LSMR algorithm, which is abbreviated as hyb-LSMR.

The organization of this paper is as follows.
In Section 2, we briefly review LSMR.
We propose the hybrid LSMR algorithm and
make an analysis on the
conditioning of inner least squares problems in Section 3.
In Section 4, we make a theoretical analysis on the stopping tolerance for
inner least squares problems.
Numerical experiments are presented in Section 5.
Finally, we conclude the paper in Section 6.

\section{Golub-Kahan bidiagonalization process and LSMR}\label{sec2}
The lower Lanczos bidiagonalization process is presented by Paige and Saunders
\cite{lsqr1982} which is a variant of the upper bidiagonalization
Lanczos process due to Golub and Kahan \cite{golub65}.
Given the initial vectors $\beta_1=\|b\|$, $p_1=b/\beta_1$, and $\alpha_1q_1=A^Tp_1$,
for $i=1, 2, 3, \cdots$, the Golub-Kahan
iterative bidiagonalization computes
\begin{eqnarray}
\beta_{i+1}p_{i+1}&=&Aq_i-\alpha_ip_{i},\label{bid2}\\
\alpha_{i+1}q_{i+1}&=&A^Tp_{i+1}-\beta_{i+1}q_{i}, \label{bid1}
\end{eqnarray}
where $\beta_{i+1}\geq 0$ and $\alpha_{i+1}\geq 0$
are normalization constants chosen so that $\|q_i\|=\|p_{i+1}\|=1$.

Define matrices
\begin{equation}\label{lsqr1}
Q_k=(q_1, \ldots,q_k)\in \mathbb{R}^{n\times k},\ P_k = (p_1,\ldots,p_k)\in \mathbb{R}^{m\times k},
\end{equation}
and
\begin{equation}\label{bk}
B_k = \left(
\begin{array}{cccc}
\alpha_1 &  &  &  \\
\beta_2 & \alpha_2 &  &  \\
&\beta_3 &\ddots  &  \\
&  & \ddots  & \alpha_k \\
&&&\beta_{k+1}\\
\end{array}
\right)\in \mathbb{R}^{(k+1)\times k}.
\end{equation}
The recurrence relations (\ref{bid1}) and (\ref{bid2})
can be written of the form
\begin{eqnarray*}
 AQ_k=P_{k+1}B_{k} \ \ {\rm and} \ \ A^{T}P_{k+1}=Q_{k}B_k^T+\alpha_{k+1}p_{k+1}(e_{k+1}^{(k+1)})^T,
\end{eqnarray*}
where $e_{k+1}^{(k+1)}$ denotes the $(k + 1)$-th canonical basis vector of $\mathbb{R}^{k+1}$.
These matrices can be computed by the following $k$-step
Golub-Kahan bidiagonalization process.
\begin{algorithm}[htbp]
\caption{\ $k$-step Golub-Kahan bidiagonalization process.}
\label{alg1}\label{alg2}
\begin{algorithmic}[1]
\STATE Take $p_1 = b/\|b\|\in \mathbb{R}^{m}$, and define $\beta_1q_0 = 0$.\\
\STATE For $j = 1, 2, \cdots, k$\\
 $r = A^Tp_j - \beta_jq_{j-1}$\\
  $\alpha_j = \|r\|$; ~$q_j = r/\alpha_j$\\
  $s = Aq_j - \alpha_jp_j$\\
  $\beta_{j+1} = \|s\|$; ~$ p_{j+1} = s/\beta_{j+1}$.\\
\end{algorithmic}
\end{algorithm}

LSMR \cite{bjorck2015numerical,hnety} is mathematically equivalent to MINRES \cite{paige1975solution} 
 applied to the normal equation 
\begin{equation}\label{normal}
A^TAx=A^Tb
\end{equation}
of \eqref{eq1}, and it solves
\begin{equation}\label{nor}
\|A^T(b-Ax_k)\|=\min_{x\in \mathcal{V}_k^R}\|A^T(b-Ax)\|
\end{equation}
for the iterate 
\begin{equation}
x_k=Q_ky_k \ {\rm with} \ y_k=\arg\min_{y\in\mathbb{R}^k}\left\|\left(\begin{array}{c}
    B_k^TB_k  \\
\alpha_{k+1}\beta_{k+1}e_k^{(k)} 
\end{array}\right)y-\alpha_1\beta_1e_1^{(k+1)}\right\|,  
\end{equation}
where $\mathcal{V}_k^R=\mathcal{K}_k(A^TA,A^Tb)$ is the
$k$ dimensional Krylov subspace generated by the matrix $A^TA$ and the vector $A^Tb$.

From the $k$-step Golub-Kahan
bidiagonalization process, it follows
\begin{equation}\label{pro_2}
Q_{k+1}^TA^TAQ_k=\left(B_k^TB_k,\alpha_{k+1}\beta_{k+1}e_k^{(k)}\right)^T.
\end{equation}
Noting $Q_{k+1}^TA^Tb=\alpha_1\beta_1e_1^{(k+1)}$, from \eqref{pro_2}, we have
\begin{eqnarray}\label{xk}
x_k&=&Q_k(Q_{k+1}^TA^TAQ_k)^{\dagger}Q_{k+1}^TA^Tb\\
&=&\left(Q_{k+1}\left(\begin{array}{c}
    B_k^TB_k  \\
\alpha_{k+1}\beta_{k+1}e_k^{(k)} 
\end{array}\right)Q_k^T\right)^{\dagger}A^Tb\\
&=&
Q_k\left(\begin{array}{c}
    B_k^TB_k  \\
\alpha_{k+1}\beta_{k+1}e_k^{(k)} 
\end{array}\right)^{\dagger}Q_{k+1}^TA^Tb\\
&=&
\alpha_{k+1}\beta_{k+1}Q_k\left(\begin{array}{c}
    B_k^TB_k  \\
\alpha_{k+1}\beta_{k+1}e_k^{(k)} 
\end{array}\right)^{\dagger}e_1^{k+1}
\end{eqnarray}
Therefore, LSMR solves the problem
\begin{equation}\label{pro}
\min\left\|\left(Q_{k+1}\left(\begin{array}{c}
    B_k^TB_k  \\
 \alpha_{k+1}\beta_{k+1}e_k^{(k)} 
\end{array}\right)Q_k^T\right)x-A^Tb\right\|
\end{equation}
for $x_k$ starting with $k=1$,
where the projection $$Q_{k+1}\left(B_k^TB_k,\alpha_{k+1}\beta_{k+1}e_k^{(k)}\right)^TQ_k^T=Q_{k+1}Q_{k+1}^TA^TAQ_kQ_k^T,$$ 
a rank-$k$ approximation to $A^TA$.
Substitutes $A^TA$ in the normal equation \eqref{normal} of underlying ill-posed problem \eqref{eq1} with it's a rank-$k$ approximation, we have \eqref{pro}.

Jia in \cite{jia20203} points out that the semi-convergence of LSMR
must occur no sooner than that of LSQR and
the regularizing effects of LSMR are not inferior to those
of LSQR and the best regularized solutions by LSMR
are at least as accurate as those by LSQR which means 
LSMR has the same regularization ability as that of LSQR. Jia also points out that LSMR has full regularization for severely or moderately ill-posed problems with suitable $\rho>1$ or $\alpha>1$.
About the accuracy of the rank-$k$ approximation and
the relationship between the regularizing effects of LSMR and those of LSQR and comparison of the regularization ability of the two methods see \cite{jia20203} for more details.

This means for mildly ill-posed and some
severely or moderately ill-posed problems with 
$\rho \leq 1$ and $\alpha \leq 1$
LSMR does not have full regularization.
Therefore, 
we consider to employ the general-form regularization term to the projected problems.

\section{Hybrid LSMR algorithm}\label{sec4}

According to the above analysis,
we apply the general-form regularization
term to \eqref{pro}. Define $$B_{proj}=Q_{k+1}\left(\begin{array}{c}
    B_k^TB_k  \\
\alpha_{k+1}\beta_{k+1}e_k^{(k)} 
\end{array}\right)Q_k^T.$$ 
This means we solve the following problems
\begin{equation}\label{hyb}
\min_{x\in\mathbb{S}}\|Lx\| 
\ {\rm subject \ to} \
\mathbb{S}=\left\{x|\left\|B_{proj}x-A^Tb\right\|=\min\right\}, 
\end{equation}
starting with $\ k=1, 2\cdots.$

Regarding the hyb-LSMR solution $x_{L,k}$,
we can establish the following results.
\begin{theorem}\label{thm}
Let $Q_{k}$ and $Q_{k+1}$ be defined in \eqref{lsqr1}.
Then the solution to \eqref{hyb} can be written as
\begin{equation}\label{xLk}
x_{L,k} =x_k -\left(L(I_n-Q_kQ_k^T)\right)^{\dagger} Lx_k.
\end{equation}
\end{theorem}
\begin{proof}
From Eld$\acute{e}$n \cite{elden82} and \eqref{xk}, we derive
\small{
\begin{equation*}
x_{L,k} = \left(I_n - \left[L\left(I_n-B_{proj}^{\dagger}
B_{proj}\right)\right]^{\dagger}L\right)x_k.
\end{equation*}
}
Since ${B}_k$ is invertible and $Q_k$ is column orthonormal,  
it is to see that
\begin{align*}
  B_{proj}^{\dagger }= \left(Q_{k+1}\left(\begin{array}{c}
    B_k^TB_k  \\
\alpha_{k+1}\beta_{k+1}e_k^{(k)} 
\end{array}\right)Q_k^T\right)^{\dagger}=
Q_k\left(\begin{array}{c}
    B_k^TB_k  \\
\alpha_{k+1}\beta_{k+1}e_k^{(k)} 
\end{array}\right)^{\dagger}Q_{k+1}^T.
\end{align*}
Therefore
\begin{equation*}
   B_{proj}^{\dagger}B_{proj} =Q_k\left(\begin{array}{c}
    B_k^TB_k  \\
\alpha_{k+1}\beta_{k+1}e_k^{(k)} 
\end{array}\right)^{\dagger}Q_{k+1}^T
Q_{k+1}\left(\begin{array}{c}
    B_k^TB_k  \\
\alpha_{k+1}\beta_{k+1}e_k^{(k)} 
\end{array}\right)Q_k^T
=Q_kQ_k^T,
\end{equation*}
because $Q_k^TQ_k=I_k$, and $B_{proj}^{\dagger}B_{proj}=I_k$ due to 
$B_{proj}$ being column full rank.
Combining the above equations,
we obtain \eqref{xLk}.
\end{proof}


From the form of the regularization solution \eqref{xLk},
the solution can be thought of as the iterative solution \eqref{xk} modified by $\left(L(I_n-Q_kQ_k^T)\right)^{\dagger} Lx_k$.
When $L=I_n$, recalling \eqref{xk} and $Q_k^TQ_k=I_n$, it follows from \eqref{xLk} that
\begin{eqnarray*}
x_{I,k} &=&x_k -\left((I_n-Q_kQ_k^T)\right)^{\dagger} x_k\\
&=&(I_n-(I_n-Q_kQ_k^T))x_k\\
&=&Q_kQ_k^Tx_k=Q_kQ_k^TQ_k(Q_{k+1}^TA^TAQ_k)^{\dagger}Q_{k+1}^TA^Tb\\
&=&Q_k(Q_{k+1}^TA^TAQ_k)^{\dagger}Q_{k+1}^TA^Tb=x_k
\end{eqnarray*}
with $\left((I_n-Q_kQ_k^T)\right)^{\dagger}=I_n-Q_kQ_k^T$
because $(I_n-Q_kQ_k^T)$ is an orthogonal projection.
This means, when $L=I_n$, the regularized solutions solved by our hybrid algorithms are
the ones to standard-form regularization problems.

Now we consider the product of the Moore-Penrose inverse
and vector in \eqref{xLk}.
Let $z_k = (L(I_n-Q_kQ_k^T))^{\dagger}Lx_k$.
Then $z_k$ is the solution to the following least squares problem
\begin{equation}\label{eq10}
\min_{z\in \mathbb{R}^n}\left\|\left(L(I_n-Q_kQ_k^T)\right)z - Lx_k\right\|.
\end{equation}
Because of the large size of $L(I_n-Q_kQ_k^T)$,
we suppose that the problems \eqref{eq10}
can only be solved by iterative algorithms.
We will discuss to use the LSQR algorithm \cite{lsqr1982}
to solve the problems. In order to take fully advantage of
the sparsity of $L$ itself and reduce the overhead of computation
and the storage memory, it is critical to avoid forming the dense matrix
$L(I_n-Q_kQ_k^T)$ explicitly within LSQR.
Notice that the only action of $L(I_n-Q_kQ_k^T)$ in
LSQR is to form the products
of it and its transpose with vectors. We propose \cref{alg3},
which efficiently implements the Lanczos bidiagonalization process
without forming $L(I_n-Q_kQ_k^T)$ explicitly.

\begin{algorithm}[htbp]
	\caption{$\widehat k$-step Lanczos bidiagonalization process on
		$L(I_n -Q_kQ_k^T)$.}
    \label{alg3}
	\begin{algorithmic}[1]
		\STATE Taking $\beta_1\widehat{u}_1= Lx_k$, $\beta_1=\|Lx_k\|$, $w_1= L^T\widehat{u}_1$,
	   and define $\widehat{\beta}_1\widehat{v}_0 = 0$.
		\STATE For $j=1, 2, \ldots, \widehat k$ \\
		$\widehat{r}=w_j-Q_k(Q_k^Tw_j)-\widehat{\beta}_j\widehat{v}_{j-1}$\\
		$\widehat{\alpha}_j=\|\widehat{r}\|$;~~$\widehat{v}_j=\widehat{r}/\widehat{\alpha}_j$;
         ~$g_{j} = Q_k^T\widehat{v}_{j}$\\
		$\widehat{s}=L\widehat{v}_j-L(Q_kg_j)-\widehat{\alpha}_j\widehat{u}_j$ \\
		$\widehat{\beta}_{j+1} = \|\widehat{s}\|$;~~$\widehat{u}_{j+1} = \widehat{s}/\widehat{\beta}_{j+1}$;
		~$w_{j+1} = L^T\widehat{u}_{j+1}$\\
	\end{algorithmic}
\end{algorithm}

Next, we consider to solve \eqref{eq10} by using LSQR.
Define
\begin{equation}\label{Q}
Q=\left(
\begin{array}{cc}
Q_k & Q_k^{\perp} \\
\end{array}
\right)
\end{equation}
as an orthogonal matrix and $Q_k^{\perp}\in\mathbb{R}^{n\times (n-k)}$
an orthogonal complement of the matrix $Q_k$.
Using the notation of \eqref{Q}, it is easy to derive
\begin{equation}\label{LQ}
L(I_n -Q_kQ_k^T)=LQ_k^{\perp}(Q_k^{\perp})^T.
\end{equation}
The nonzero singular values of $LQ_k^{\perp}(Q_k^{\perp})^T$
are identical to those of $LQ_k^{\perp}$, 
due to the orthogonality of $Q_k^{\perp}$.
Therefore, we obtain
\begin{equation}\label{condi}
\kappa(L(I_n-Q_kQ_k^T))=\kappa(LQ_k^{\perp}(Q_k^{\perp})^T)
=\kappa(LQ_k^{\perp}).
\end{equation}

We now establish the result about how the conditioning of (\ref{eq10})
changes as $k$ increasing with these notations.
\begin{theorem}\label{thm3}
Let the matrix $Q_k^{\perp}$ be defined in (\ref{Q}) and $L\in\mathbb{R}^{p\times n}$.
When $p\geq n-k$, we obtain
\begin{equation}\label{eq3}
\kappa(LQ_k^{\perp})\geq\kappa(LQ_{k+1}^{\perp}),\quad k = 2, 3, \cdots, n-1,
\end{equation}
that is,
\begin{equation}\label{eq2}
\kappa(L(I_n -Q_kQ_k^T))\geq\kappa(L(I_n -Q_{k+1}Q_{k+1}^T)),
\quad k = 2, 3, \cdots,n-1.
\end{equation}
\end{theorem}

\begin{proof}
By using the lemma from \cite[pp.78]{golub13}, it can be obtained 
\begin{eqnarray*}
\sigma_{\max}(LQ_k^{\perp})\geq \sigma_{\max}(LQ_{k+1}^{\perp}),\label{1}\\
\sigma_{\min}(LQ_k^{\perp})\leq \sigma_{\min}(LQ_{k+1}^{\perp}).\label{2}
\end{eqnarray*}
From these equations, we obtain
\begin{equation*}
\kappa(LQ_k^{\perp})=\frac{\sigma_{\max}(LQ_k^{\perp})}{\sigma_{\min}(LQ_k^{\perp})}
\geq \frac{\sigma_{\max}(LQ_{k+1}^{\perp})}{\sigma_{\min}(LQ_{k+1}^{\perp})}
=\kappa(LQ_{k+1}^{\perp}),
\end{equation*}
which implies \eqref{eq3}. 
We directly obtain \eqref{eq2} from \eqref{eq3} by noticing \eqref{LQ}.
\end{proof}

The theorem indicates that, when applied to solving \eqref{eq10},
the LSQR algorithm generally converges faster with $k$;
see \cite[p.~291]{bjorck96}.
Particularly, in exact arithmetic,
LSQR will find the exact solution $z_k$ of \eqref{eq10}
after at most $n-k$ iterations.
Therefore, the inner iteration
is faster convergence with $k$ becoming larger,
which means the LSQR method solving \eqref{eq10}
convergent faster with $k$ becoming larger.

Next, we present the hyb-LSMR, named as \cref{alg:plsmr}.

\begin{algorithm}[htbp]
\caption{(hyb-LSMR) 
\ Given $A\in \mathbb{R}^{m\times n}$,
$L\in\mathbb{R}^{p\times n}$ and $b\in\mathbb{R}^n$, compute the solution
$x_{L,k}$.}
\label{alg:plsmr}
\begin{algorithmic}[1]
\STATE Use Algorithm \ref{alg1}
to compute the projection $B_{proj}$.
\STATE Compute $x_k$ by \eqref{xk}.
\STATE Compute $z_k$ for \eqref{eq10} by LSQR and \cref{alg3}.
\STATE Compute the solution $x_{L,k}=x_k-z_k$.
\end{algorithmic}
\end{algorithm}

About the stopping tolerance of the inner least squares problems,
what we have to highlight is that
it is generally enough to set
$tol=10^{-6}$ in LSQR with Algorithm~\ref{alg2} at step 3 of Algorithms~\ref{alg:plsmr}.
Moreover,  $tol=10^{-6}$
is generally well conservative and larger $tol$ can be used, so
that LSQR with \cref{alg2} uses fewer iterations to achieve the convergence
and the hyb-LSMR algorithm
is more efficient; see \cite{yang} and \cite{jiayang2018} for more details.

\section{Numerical experiments}
In this section, we report numerical experiments
on several problems, including One-dimensional and
two-dimensional problems,
to demonstrate that the proposed new hybrid algorithms
work well and that the regularized solutions obtained by 
hyb-LSMR are at least as accurate
as those obtained by JBDQR,
and the proposed new method is considerably more efficient than JBDQR.

Table~\ref{tab1} lists all test problems,
which are from the
regularization toolbox \cite{nagy2004}. 
The one-dimensional problems are of the order $m=n=1,000$.
The two-dimensional problems are from \cite{nagy2004} with an order 
of $m=n=65, 536$. 
We denote the relative noise level
$$
\varepsilon = \frac{\|e\|}{\|b_{true}\|}.
$$

All the computations are carried out in Matlab R2019a 64-bit on
11th Gen Intel(R) Core(TM) i5-1135G7 2.40GHz processor
and 16.0 GB RAM with the machine precision
$\epsilon_{\rm mach}= 2.22\times10^{-16}$ under the Miscrosoft
Windows 10 64-bit system.

\begin{table}[ht]
\caption{The description of test problems.}\label{tab1}
\begin{tabular}{@{}lll@{}}
\toprule
 Problem        & Description                                & Size of $m, n$ \\
\midrule
     {shaw}     & one dimensional image restoration model    & $m=n=1,000$\\
     {baart}    & one dimensional gravity surveying problem  &  $m=n=1,000$\\
     {heat}     & Inverse heat equation                   &  $m=n=1,000$\\
     {gravity}  & one dimensional gravity surveying problem  &  $m=n=1,000$\\
     {deriv2}   & Computation of second derivative        &  $m=n=1,000$\\
     grain      & Two dimensional image deblurring      & $m=n=65, 536$\\
     text2    & spatially variant Gaussian blur      & $m=n=65, 536$\\
     {satellite} &spatially invariant atmospheric turbulence        & $m=n=65, 536$ \\
     {GaussianBlur440} & spatially invariant Gaussian blu      & $m=n=65, 536$\\
\bottomrule
\end{tabular}
\end{table}

Let $x^{reg}$ be the regularized solution obtained by
algorithms. We use the relative error
\begin{equation}\label{rel}
\frac{\|L(x^{reg}-x_{true})\|}{\|Lx_{true}\|}
\end{equation}
to plot the convergence curve of each algorithm with respect to $k$,
which is more instructive and suitable to use
the relative error (\ref{rel}) in the general-form regularization
context other than the standard relative error of $x_{reg}$;
see \cite[Theorems 4.5.1-2]{hansen98} for more details.
In the tables to be presented,
we will list the smallest relative errors and the total outer iterations
required to obtain the smallest relative errors in the braces.
We also will list the total CPU time which is counted in seconds
by the Matlab built-in commands
{\sf tic} and {\sf toc}
and the corresponding total outer iterations.
For the sake of length, 
we only display the noise level $\varepsilon =
\ 10^{-2}$ in Tables \ref{tab2} and \ref{tab3}.

For our new hybrid algorithms hyb-LSMR,
we use the Matlab built-in function {\sf lsqr}
with Algorithm~\ref{alg1} avoiding forming the dense matrices $L(I_n-Q_kQ_k^T)$
and $L(I_n-Q_{k+1}Q_{k+1}^T)$ explicitly
to compute (\ref{eq10})
with the default stopping tolerance $tol=10^{-6}$.
For the JBDQR algorithm \cite{jiayang2020},
we use the same function with the same stopping tolerance
to solve
the inner least squares problems.

The regularization matrix is chosen as
\begin{equation}\label{l3}
L=\left(
    \begin{array}{c}
      I_N\otimes L_1 \\
      L_1 \otimes I_N \\
    \end{array}
  \right),
\end{equation}
where  
\begin{equation}\label{l1}
L_1 = \left(
        \begin{array}{ccccc}
          1 & -1 &  &  &  \\
           & 1 & -1 &  &  \\
           &  & \ddots & \ddots &  \\
             &  &  & 1  & -1\\
        \end{array}
      \right)\in \mathbb{R}^{(n-1)\times n}
\end{equation}
is the scaled discrete approximation of the first
derivative operator in the two dimensional case incorporating no assumptions
on boundary conditions; see \cite[Chapter 8.1-2]{hansen10}.

\begin{table}[ht]
\caption{The relative errors and the optimal
regularization parameters in the braces for test problems in Table~\ref{tab1}.}\label{tab2}
\begin{minipage}[t]{1\textwidth}
\begin{tabular*}{\textwidth}{@{\extracolsep\fill}lcccccc}
\toprule%
& \multicolumn{2}{@{}c@{}}{$\varepsilon=10^{-2}$}  \\\cmidrule{2-3}%
&{hyb-LSMR}    &JBDQR     \\
\midrule
 {shaw}     &0.1630(8)     &  0.1743(4)     \\
 {baart}    &0.5492(3)      &0.5976(1)         \\
 {heat}     & 0.2697(16)    &0.2568(17)         \\
 {gravity}     & 0.3413(9)     &1.0341(1)  \\
{grain}     &0.7017(71)     &0.7872(48)        \\
{text2}       &0.9704(84)     &0.9791(86)         \\
{satellite} &0.9318(113)     &0.9327(98)      \\
{GaussianBlur440}   &0.9519(62)     &0.9515(116)       \\
\bottomrule
\end{tabular*}
\end{minipage}
\end{table}

From Table \ref{tab2},
we observe that for all test problems,
the best regularized solution by {hyb-LSMR} is
at least as accurate as, and often considerably more accurate than,
that by JBDQR;
see, e.g., the results on test problems shaw, baart, gravity, and grain.
In particular, for gravity, JBDQR fails, while our new hybrid algorithm performs well and obtains a highly accurate regularized solution.

\begin{table}[ht]
\caption{The CPU time of {hyb-LSMR} and JBDQR,
the ratio of the CPU time of JBDQR to that of hyb-LSMR  ({times}),
and the total outer iterations ({iteration})
for test problems in Table~\ref{tab1}.}\label{tab3}
\begin{minipage}[t]{1\textwidth}
\begin{tabular*}{\textwidth}{@{\extracolsep\fill}lcccc}
\toprule%
\end{tabular*}
\end{minipage}
\begin{minipage}[t]{1\textwidth}
\begin{tabular*}{\textwidth}{@{\extracolsep\fill}lcccc}
& \multicolumn{4}{@{}c@{}}{$\varepsilon=10^{-2}$ } \\\cmidrule{2-5}%
&{\ hyb-LSMR}    &JBDQR     &{ times}  &{ iteration}\\
\midrule
 {shaw}     &0.3419   &3.4139 &9.9834  	   &28   \\
 {baart}    &0.30367    &4.0828   &13.4445&28  \\
 {heat}     &0.3854  &3.7596   &9.7547 &28 \\
 {gravity}     &0.3373  &3.7329   &11.3716  &28\\
{grain}   	&345.0749	 &7629.2553 &22.1089 &200	  \\
{text2}    &1404.5045 &5719.7726		&4.0724 &200	\\
{satellite}  &958.6520 &4647.4674   &4.8479 &200\\
{GaussianBlur440}  &287.5606 &976.4092	&3.3954 &200	\\
\bottomrule
\end{tabular*}
\end{minipage}
\end{table}

Based on the analysis of the JBD process in Section \ref{sec2} and the comments in Section \ref{sec4}, at the same number of outer iterations, hyb-LSMR can be significantly cheaper than JBDQR. As shown in Table \ref{tab3}, for each test problem, the CPU time of {hyb-LSMR} is notably less than that of JBDQR for the same number of outer iterations. For one-dimensional problems, the CPU time for our new hybrid LSMR algorithm is under half a second, whereas JBDQR takes more than three seconds. For two-dimensional problems, hyb-LSMR completes in under one thousand seconds for all cases except text2, while JBDQR requires over one thousand seconds, with grain taking more than seven thousand seconds.

\begin{figure}{ht}
 \begin{minipage}{0.48\linewidth}
   \centerline{\includegraphics[width=6.0cm,height=4cm]{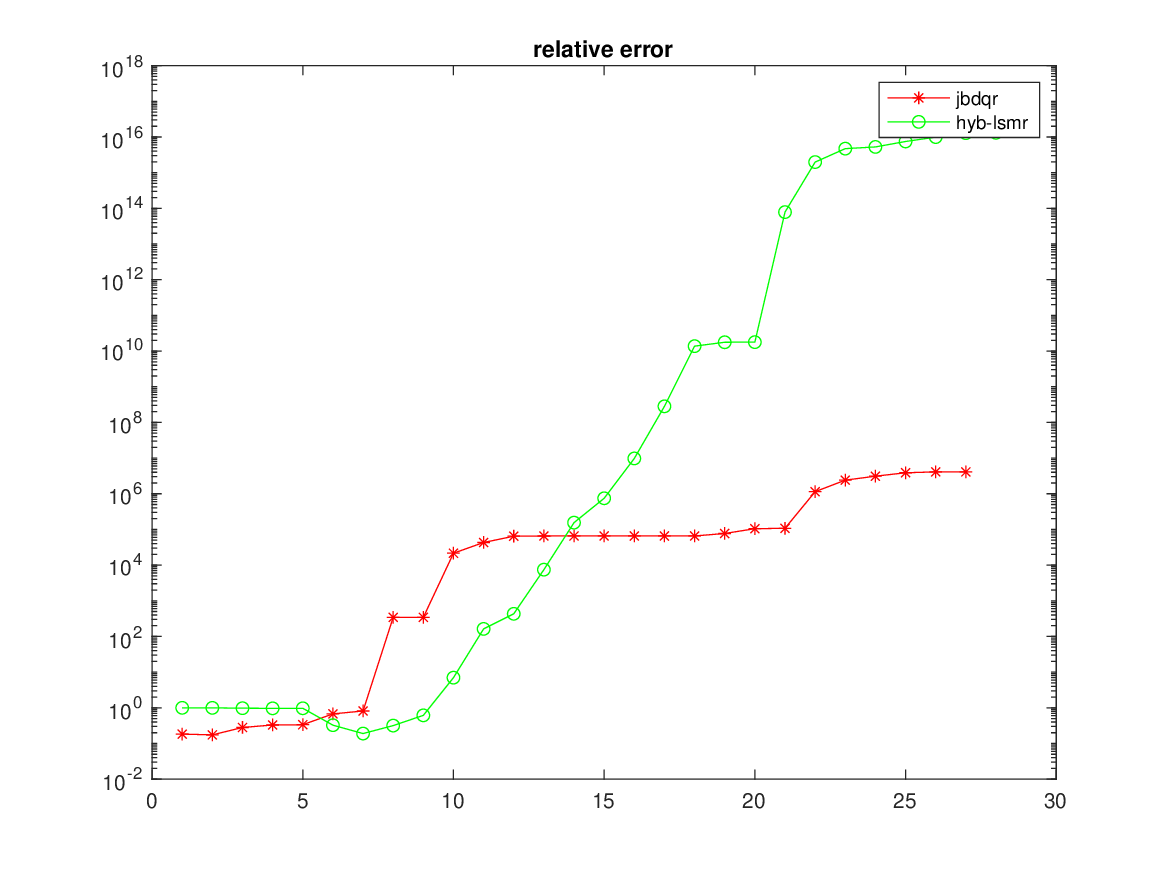}}
   \centerline{(a)}
 \end{minipage}
 \hfill
 \begin{minipage}{0.48\linewidth}
   \centerline{\includegraphics[width=6.0cm,height=4cm]{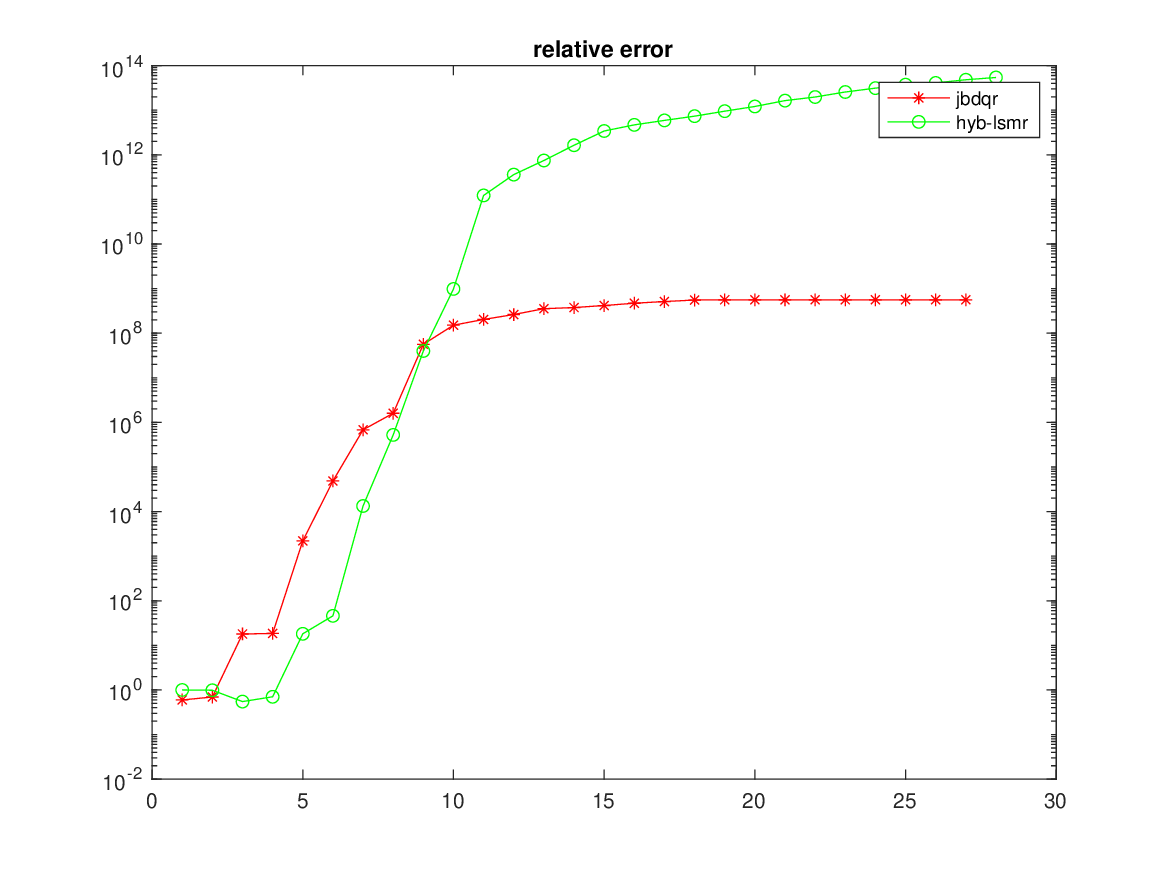}}
   \centerline{(b)}
 \end{minipage}

 \vfill
 \begin{minipage}{0.48\linewidth}
   \centerline{\includegraphics[width=6.0cm,height=4cm]{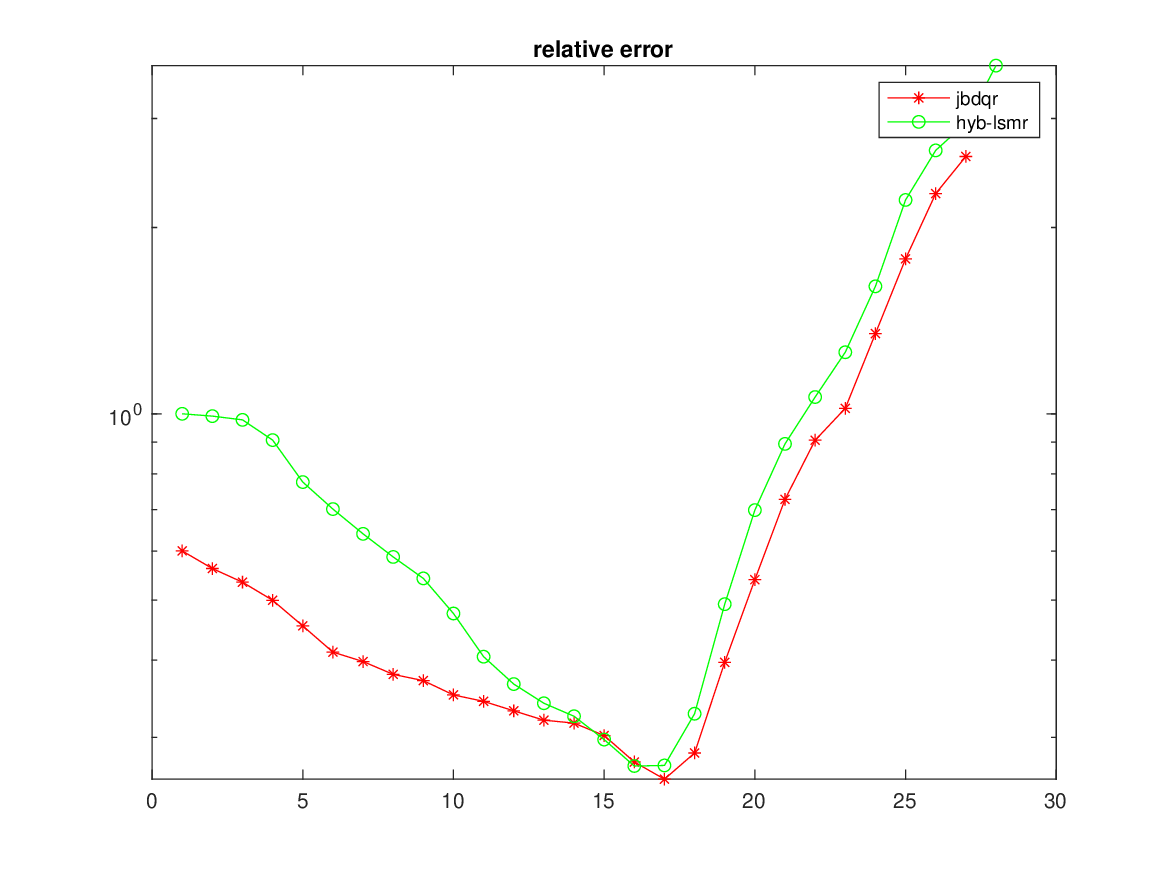}}
   \centerline{(c)}
 \end{minipage}
 \hfill
 \begin{minipage}{0.48\linewidth}
   \centerline{\includegraphics[width=6.0cm,height=4cm]{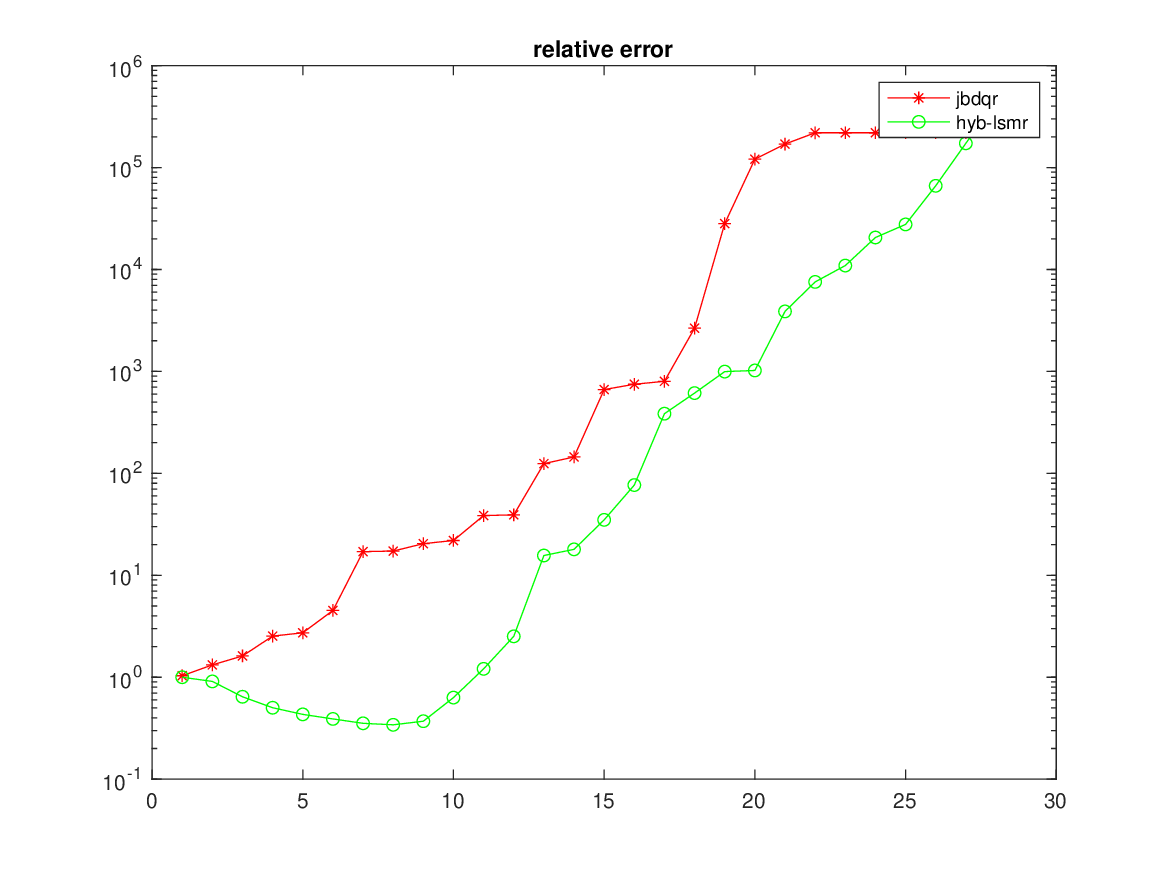}}
   \centerline{(d)}
 \end{minipage}
 \hfill
\begin{minipage}{0.48\linewidth}
  \centerline{\includegraphics[width=6.0cm,height=4cm]{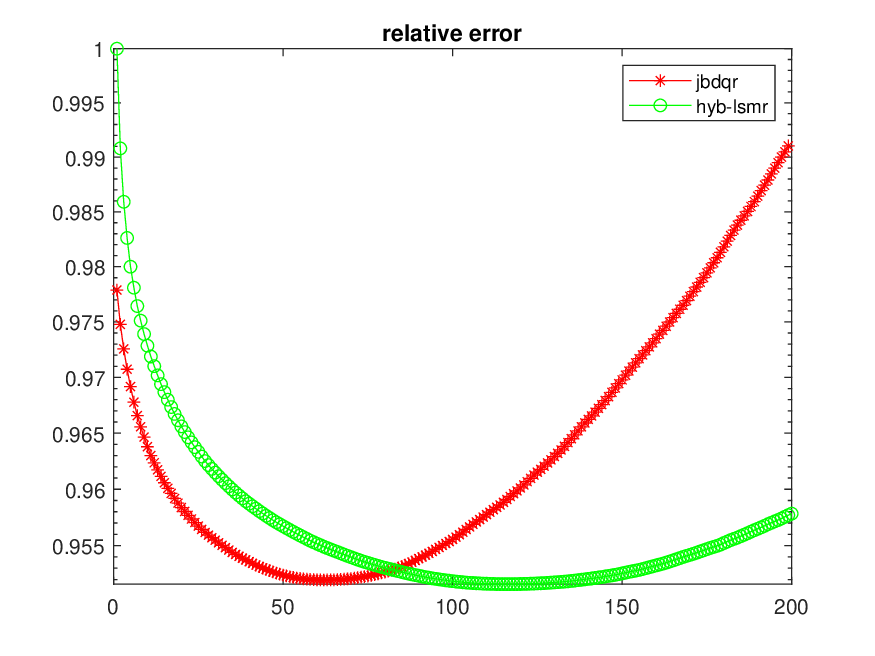}}
  \centerline{(e)}
\end{minipage}
\hfill
\begin{minipage}{0.48\linewidth}
  \centerline{\includegraphics[width=6.0cm,height=4cm]{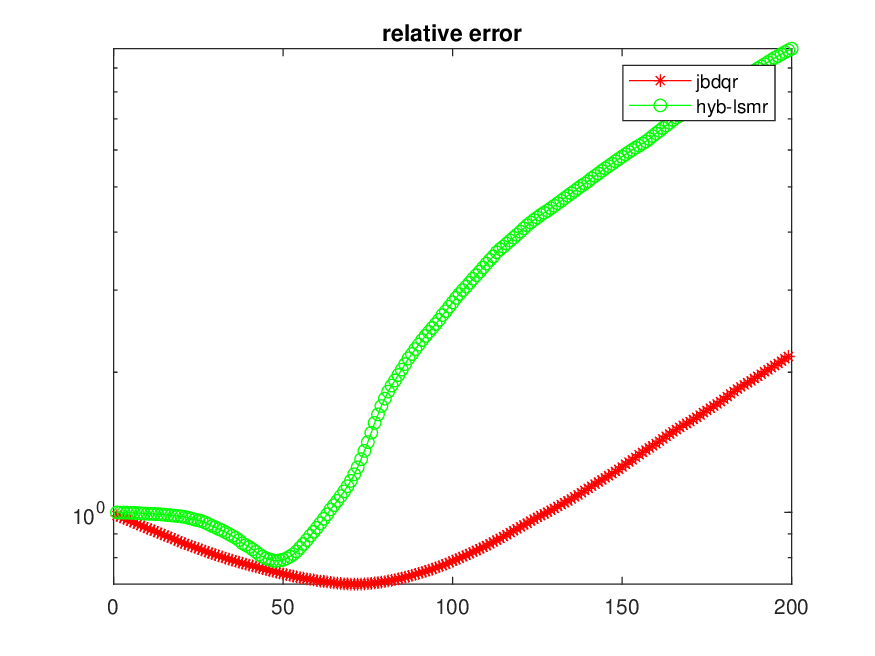}}
  \centerline{(f)}
\end{minipage}
\vfill
\begin{minipage}{0.48\linewidth}
  \centerline{\includegraphics[width=6.0cm,height=4cm]{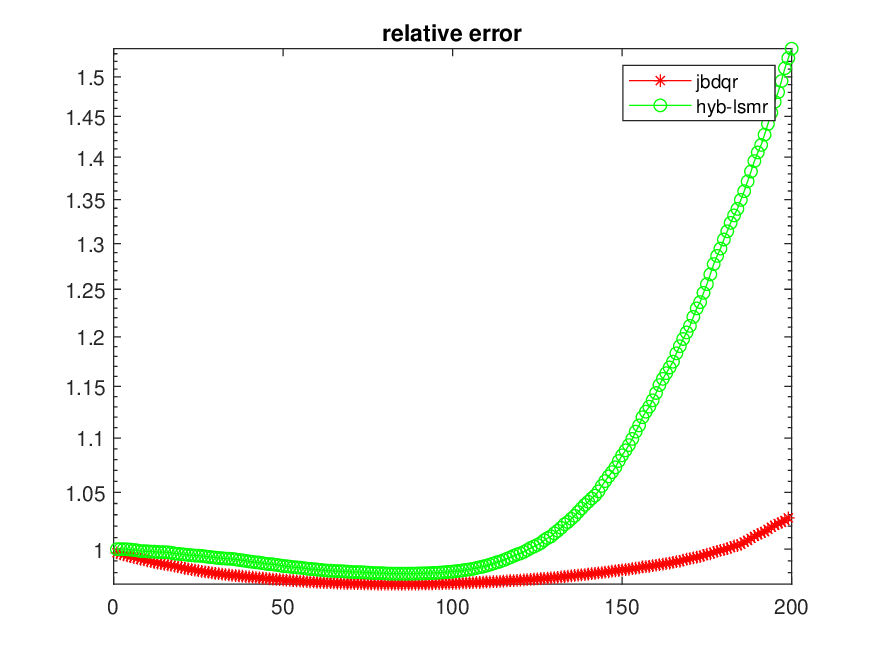}}
  \centerline{(g)}
\end{minipage}
\hfill
\begin{minipage}{0.48\linewidth}
  \centerline{\includegraphics[width=6.0cm,height=4cm]{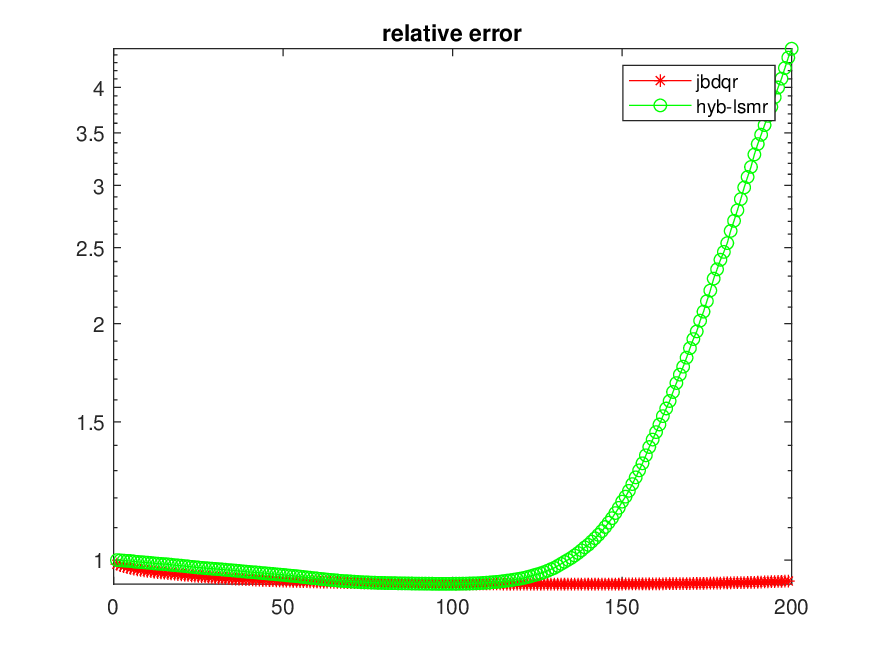}}
  \centerline{(h)}
\end{minipage}
\caption{The relative error of {hyb-LSMR} and JBDQR
 with $\varepsilon=10^{-2}$: 
 (a) { baart}; (b) { shaw}; (c) { heat}; (d) { gravity};
(e) {Gaussianblur440}; (f) {grain}; (g) {text2}; (h) {satellite}.}
\label{fig2}
\end{figure}

In Figure \ref{fig2} we display the convergence processes of {hyb-LSMR} and JBDQR for $\varepsilon=10^{-2}$.
We can see that the best regularized solutions by {hyb-LSMR} are at least as accurate as, sometime
more accurate than,
the counterparts by JBDQR except {grain},
but hyb-LSMR takes much less time than JBDQR does for grain,
which we can see from table \ref{tab2}.
Moreover, as the figure shows, for every test problem,
both algorithms under consideration exhibit semi-convergence \cite{hansen98, hansen10}:
the convergence curves of the three algorithms first decrease with $k$,
then increase.
This means the iterates converge to $x_{true}$
in an initial stage; afterwards the iterates start to diverge from $x_{true}$.
It is worth mentioning that it is proved \cite{jiayang2020} that the
JBDQR iterates take the form of ﬁltered GSVD
expansions, which shows that JBDQR have the semi-convergence property.
These results indicate that both algorithms exhibt the semi-convergence property.

 In summary, for all test problems, hyb-LSMR outperforms JBDQR in both accuracy and efficiency. Therefore, it serves as a strong alternative to JBDQR.

\begin{figure}
\begin{minipage}{0.48\linewidth}
  \centerline{\includegraphics[width=6.0cm,height=4cm]{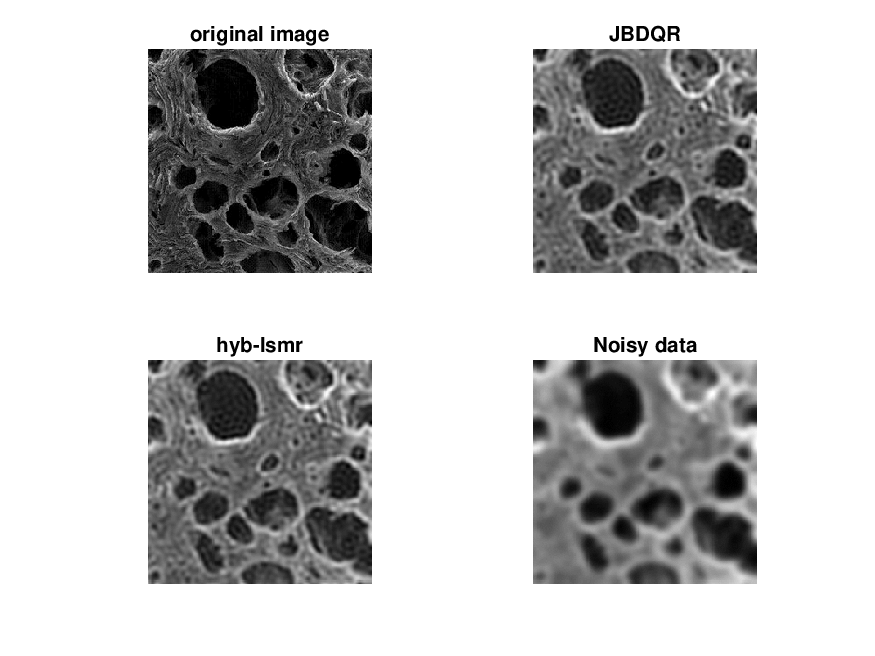}}
  \centerline{(a)}
\end{minipage}
\hfill
\begin{minipage}{0.48\linewidth}
  \centerline{\includegraphics[width=6.0cm,height=4cm]{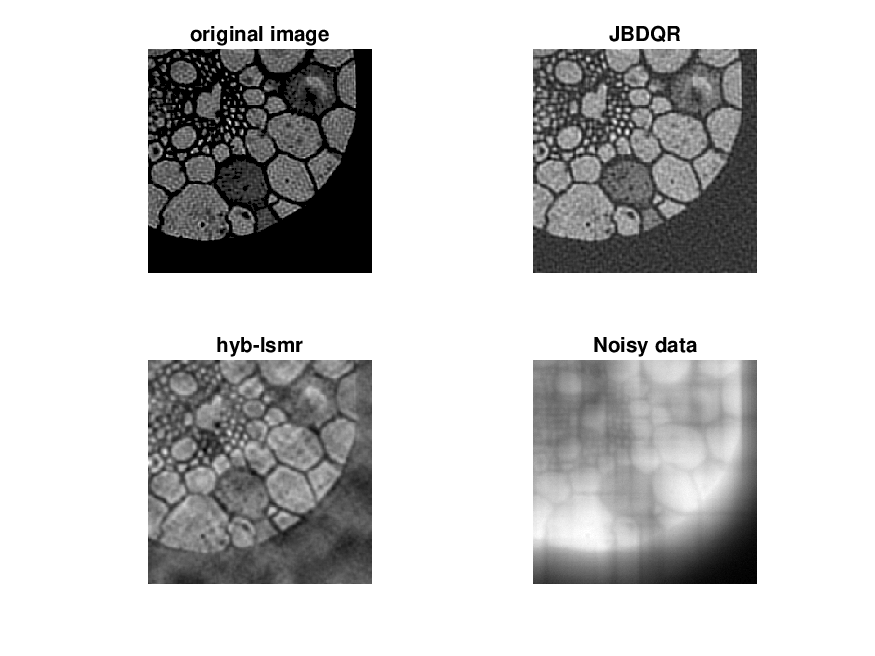}}
  \centerline{(b)}
\end{minipage}
\vfill
\begin{minipage}{0.48\linewidth}
  \centerline{\includegraphics[width=6.0cm,height=4cm]{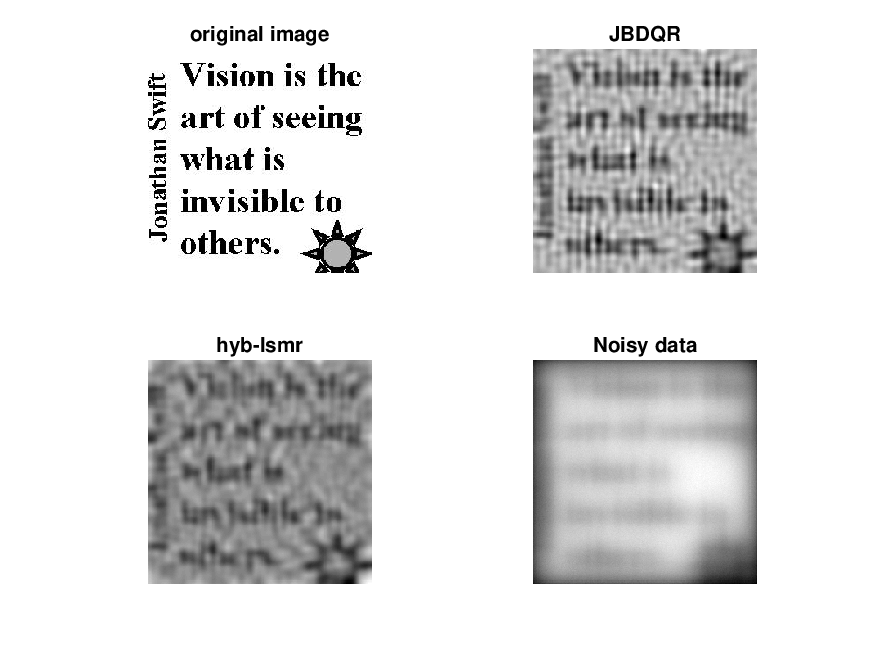}}
  \centerline{(c)}
\end{minipage}
\hfill
\begin{minipage}{0.48\linewidth}
  \centerline{\includegraphics[width=6.0cm,height=4cm]{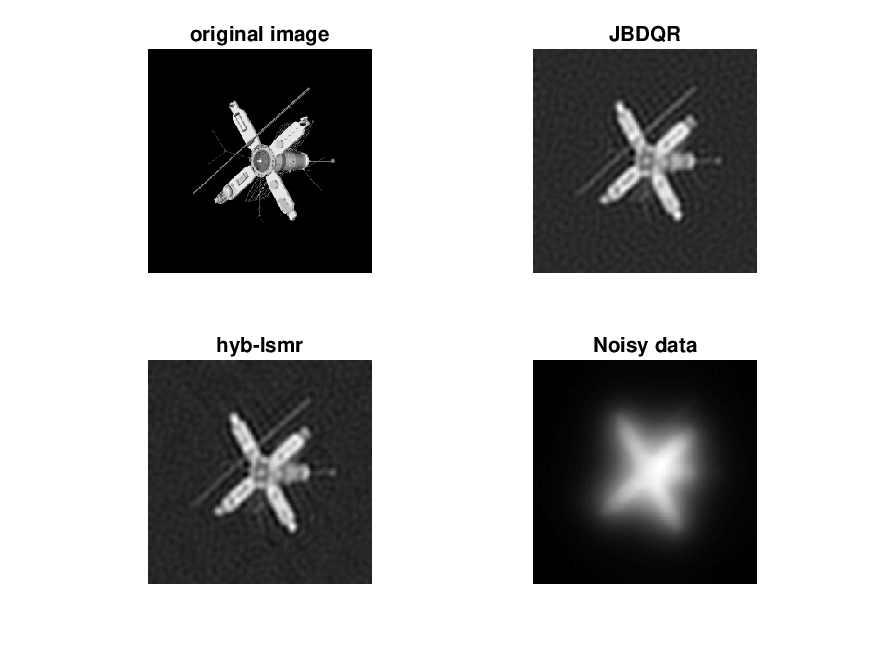}}
  \centerline{(d)}
\end{minipage}
\caption{The exact images and the reconstructed images for the four two dimensional test problems
with $\varepsilon=10^{-2}$ and $L$ defined in (\ref{l3}):
(a) {Gaussianblur440}; (b) {grain};(c) text2;  (d) {satellite}.}
\label{fig3}
\end{figure}

The exact images and the
reconstructed images for the four test problems
with $\varepsilon=10^{-2}$ and $L$ defined by (\ref{l3})
are displayed in Figure \ref{fig3}.
Clearly, the reconstructed
images by {hyb-LSMR} are at least as sharp as those
by JBDQR, especially for GaussianBlur440 and satellite .

\section*{Declarations}

%
\begin{itemize}
\item This study was funded by Zhejiang A and F University.
(No. 203402000401).
\item Not applicable.
\item Data availability
\item Materials availability
\item Code availability
\end{itemize}
%




\bibliographystyle{siamplain}
\bibliography{hyblsmr}

\end{document}